\documentclass[11pt,a4paper]{amsart}
\usepackage{amsmath,amssymb, amsbsy}
\usepackage{enumerate}
\usepackage{hyperref}
\usepackage{tikz}

\usepackage{xcolor}

\newtheorem{theorem}{Theorem}[section]
\newtheorem{proposition}[theorem]{Proposition}
\newtheorem{lemma}[theorem]{Lemma}
\newtheorem{corollary}[theorem]{Corollary}

\theoremstyle{definition}
\newtheorem{definition}[theorem]{Definition}
\newtheorem{example}[theorem]{Example}

\newtheorem{remark}[theorem]{Remark}

\newcommand{\abs}[1]{\left|#1\right|}

\usepackage{bm}             
\usepackage{float}

\newcommand{\R}{{\mathbb R}}

\renewcommand{\epsilon}{\varepsilon}
\renewcommand{\phi}{\varphi}

\newcommand{\norm}[1]{\left\|#1\right\|}
\newcommand{\scalar}[2]{\left\langle#1,#2\right\rangle}
\newcommand{\dx}[1][x]{\,\mathrm{d}#1}
\newcommand{\dt}{\dx[t]}
\newcommand{\pd}[2]{\frac{\partial #1}{\partial #2}}

\DeclareMathOperator*{\argmin}{\mathrm{argmin}}

\DeclareMathOperator*{\esssup}{\mathrm{ess\,sup}}
\DeclareMathOperator*{\essinf}{\mathrm{ess\,inf}}

\renewcommand{\div}{\mathrm{div}}
\DeclareMathOperator{\convhull}{\overline{conv}}

\DeclareMathOperator{\supp}{\mathrm{supp}}

\newcommand{\parbd}{\partial_{\operatorname{par}}}
\newcommand{\laplace}{\Delta}

\newcommand{\piku}{\Pi_K \bm u}

\begin{document}

\title[Convex hull property]{Convex hull property for elliptic and parabolic systems of PDE}

\author{Anton\'in \v Ce\v s\'ik}
\address{Faculty of Mathematics and Physics, Charles University, Sokolovsk\'{a} 83, 18675, Prague, Czech Republic} \email{cesik@karlin.mff.cuni.cz}


\begin{abstract}
We study the convex hull property for systems of partial differential equations. This is a generalisation of the maximum principle for a single equation. We show that the convex hull property holds for a class of elliptic and parabolic systems of non-linear partial differential equations. In particular, this includes the case of the parabolic $p$-Laplace system. The coupling conditions for coefficients are demonstrated to be optimal by means of respective counterexamples. 
\end{abstract}

\maketitle


\section{Introduction}\label{chap:intro}

Maximum principle for scalar-valued functions has been a subject of interest for a long time. Many monographs have been devoted to this study, \cite{pucciMaximumPrinciple2007} to name just one of them. Less has been investigated in its vector-valued counterpart, the \emph{convex hull property}. There have been several works which establish the convex hull property of certain variational minimization problems, relying on monotonicity and convexity \cite{bildhauerPartialRegularityClass2002}, \cite{bildhauerGeometricMaximumPrinciple2011}, \cite{leonettiMaximumPrincipleVector2005} or recently polyconvexity \cite{carozzaPolyconvexFunctionalsMaximum2023}. Still other results deal with finite element minimizers \cite{dieningConvexHullProperty2013}, approximate minimizers \cite{katzourakisMaximumPrinciplesVectorial2013} or other geometrical conditions \cite{ricceriConvexHulllikeProperty2016}.

Regarding systems of partial differential equations, the case of classical solutions of linear elliptic and parabolic systems has been dealt with extensively in the monograph \cite{kresinMaximumPrinciplesSharp2012}, where the authors discuss the maximum modulus principle for vector-valued solutions. Still another work \cite{leonardiMaximumPrinciplesQuasilinear2020} deals with quasilinear elliptic systems with special structure of coefficients.

Our area of investigation will be twofold. First, we investigate the case of weak solutions of linear elliptic systems in divergence form
\begin{equation*}
-\div(\mathbb A \nabla \bm u) = \bm 0, \quad \text{in }\Omega
\end{equation*}
where $\Omega\subset \R^N$ is a bounded Lipschitz domain and $\mathbb A\colon \Omega\to \R^{(N\times n)\times (N\times n)}$ is a fourth-order tensor of coefficients. 
In particular we focus on the possible coupling of the coefficients and show that if the structure of the coefficients is 
\begin{equation*}
\mathbb A = \mathcal A \otimes \bm a,\quad \text{i.e.}\quad \mathbb A_{ij}^{\alpha\beta}(x) = \mathcal A^{\alpha\beta} a_{ij}(x), \quad \alpha,\beta=1,\,\dots,N; i,j=1,\dots,n
\end{equation*}
where $\mathcal A\in \R^{N\times N}$ and $\bm a\colon \Omega \to \R^{n\times n}$ bounded. Note that these are the same conditions as \cite{kresinMaximumPrinciplesSharp2012}   have for the (a priori weaker) maximum modulus principle. We show that the convex hull property for this elliptic system holds, as we will see in Theorem \ref{thm:CHPellipticlin}. Sharpness of the structural assumption is demonstrate by Example \ref{ex:counterexell} where we show that the convex hull property for general elliptic systems fails even in one spatial dimension.

Second, we turn to parabolic systems 
\begin{equation*}
\partial_t \bm u  -\div (\mathbb A \nabla \bm u) +\bm b \cdot \nabla \bm u + c \bm u = 0, \quad \text{in }Q_T:=(0,T)\times \Omega
\end{equation*}
where the question of coupling of the coefficients is even more restraining,
\begin{equation*}
\mathbb A = I \otimes \bm a,\quad \text{i.e.}\quad \mathbb A_{ij}^{\alpha\beta}(x) = \mathcal \delta _{\alpha\beta} a_{ij}(x)
\end{equation*}
where $I\in \R^{N\times N}$ the identity matrix and thus $\delta_{\alpha\beta}$ the Kronecker delta. Again, this is with accord with the maximum modulus principle considerations in \cite{kresinMaximumPrinciplesSharp2012}. The necessity of the identity matrix in the decomposition is demonstrated by Example \ref{ex:counterexpara}, again in one spatial dimension, where we show failure of the convex hull property for constant diagonal non-identity matrix. 

However, we find that in this setting we are able to treat non-linearities (essentially scalar in nature) both in the leading term and in the lower-order terms, that is we alllow\begin{equation*}
\mathbb A = a_0 I \otimes \bm a,\quad \text{i.e.}\quad \mathbb A_{ij}^{\alpha\beta}(x) = a_0(x)\delta_{\alpha\beta} a_{ij}(x)
\end{equation*}
where $a_0$ is a non-negative scalar function,  and $|\bm b|\leq C \sqrt{a_0}$, and $c$ non-negative function. In this case we show that the convex hull property is satfied, for the precise statement see Theorem \ref{thm:CHPparanonlin}.

As a particular case we obtain that the convex hull property holds for the vector-valued parabolic $p$-Laplace system with $1<p<\infty$:
\begin{equation*}
-\div(|\nabla \bm u|^{p-2} \nabla \bm u) = 0 \quad \text{in } (0,T)\times \Omega.
\end{equation*}

\subsection{Types of maximum and convex hull principles}
We summarize here the types of maximum principles for scalar equations, and the analogous notions for systems. Note that by ``boundary'' we mean for elliptic equations (on $\Omega$) simply $\partial\Omega$, but for parabolic equations (on $(0,T)\times \Omega$), we mean the parabolic boundary $(\{0\}\times\Omega)\cup([0,T]\times \partial\Omega)$.
\begin{itemize}
    \item \textbf{Weak maximum principle.} The values of any solution do not exceed its maximum on the boundary.
    \item \textbf{Strong maximum principle.} If a solution attains its maximum in the interior, then it is constant. 
    \item \textbf{Maximum modulus principle.} For any solution $u$, the function $\abs u$ satisfies the weak maximum principle.
\end{itemize}
The vector-valued analogue for systems of equations is the \emph{convex hull property}. It generalizes the maximum principle\footnote{More precisely, it generalizes the maximum and minimum principle simultaneously.} in the sense that it says that the values of the solution lie ``in between'' its boundary values.
Hence for systems of equations, we have 
\begin{itemize}
    \item \textbf{Convex hull property.} The values of any solution lie in the convex hull of its boundary values.
    \item \textbf{Strong convex hull property.} If a solution attains an extremal point\footnote{Recall that $x\in M\subset \R^N$ is an \emph{extremal point of $M$} if it is not an interior point of a line segment contained in $M$.} of its boundary values, then it is constant.
    \item \textbf{Maximum modulus principle.} For any solution $\bm u$, the function $\abs {\bm u}$ satisfies the weak maximum principle.
\end{itemize}
In this paper, we are interested in the Convex hull property, in the sense above.

\subsection{Notation}
We use standard notation but to avoid any confusion, we briefly state it here.

Throughout the paper, $\Omega\subset\R^n$ will always denote a bounded Lipschitz domain. Vector-valued functions, denoted by boldface letter, usually have $N$ components and depend on $n$ variables called~$x$, that is $\bm u\colon \Omega \to \R^N$. In the time-dependent case $\bm u\colon (0,T)\times \Omega \to \R^N$ and the first variable is called $t$, here we also sometimes write $\bm u(t):= \bm u(t,\cdot)$. 
Variables under the integral sign are often omitted, e.g. $\int_\Omega \bm u \dx:= \int_\Omega \bm u(x)\dx$.
The symbol ``a.e.'' means almost everywhere either with respect to Lebsegue measure on $\Omega$ or the surface measure on $\partial\Omega$, depending on the context.
The closed convex hull of a set $M\subset \R^N$ is $\convhull M$ (note also Definition \ref{def:CHPell}).

The indices $i,j,k$ usually range over $1,\dots, n$ and are written as lower index. The indices $\alpha,\beta$ range over $1,\dots, N$ and are written as upper index. 
The Euclidean inner product is denoted  $x\cdot y$, the corresponding norm $\abs x$.
The inner product of matrices $A,B\in\R^{N\times n}$ is $A:B= \sum_{\alpha=1}^N\sum_{i=1}^n A_{i}^\alpha B_i^\alpha$, the corresponding matrix norm is $|A|$.
The positive and negative parts of a scalar are $a^+=\max (a,0)$, $a^-=\max(-a,0)$.
The (spatial) gradient of $\bm u$ is $(\nabla \bm u)_i^\alpha =  \pd{ u^\alpha}{x_i}$.
 The divergence of $\bm F\colon \Omega\to \R^{N\times n}$ is $(\div \bm F)^\alpha = \sum_{i=1}^n \pd{F^\alpha_{i}}{x_i}$.
The Laplace operator is $\Delta \bm u = \div (\nabla \bm u)$.

The $\R^N$-valued Lebesgue $L^p(\Omega;\R^N)$ and Sobolev space $W^{1,p}(\Omega;\R^N)$, the subspace with zero trace $W^{1,p}_0(\Omega;\R^N)$, and $W^{-1,p'}(\Omega;\R^N)=\left(W^{1,p}_0(\Omega;\R^N)\right)^*$ is its dual space.
The corresponding Bochner spaces are $L^p((0,T); X)$, $W^{1,p}((0,T); X)$. In all the function spaces, for $N=1$ we omit the $\R^N$ symbol.
\subsection{Example: Laplace system}
As a motivation for studying the convex hull property, let us first illustrate it on the example of the classical solution of the Laplace equation.
In this case the convex hull property can be obtained by using the maximum principle for scalar Laplace equation in every direction, as demonstrated below. 

\begin{theorem}[Convex hull property for Laplace equation in $\R^N$]
Suppose that $\bm u\in C^2(\Omega;\R^N)\cap C(\overline{\Omega};\R^N)$ solves the Laplace equation
\begin{equation*}
    \laplace \bm u = \bm 0 \quad \text{in }\Omega.
\end{equation*}
Then $\bm u(\Omega)\subset \convhull \bm u(\partial\Omega)$.
\end{theorem}
\begin{proof}
Choose a vector $h\in\R^N$, $\abs{h}=1$. Denote $v(x)=h\cdot \bm u(x)$. Then by a simple calculation, $\laplace v(x) = h\cdot \laplace \bm u(x)=0$. Thus $v \in C^2(\Omega)\cap C(\overline \Omega)$ satisfies the scalar Laplace equation and thus by the maximum principle for scalar-valued Laplace equation, $v(\Omega)\subset (-\infty, \max_{\partial  \Omega} v]$. This in fact means that
\begin{equation*}
    \bm u(\Omega)\subset \left\{x\in\R^N: h\cdot x\leq \max_{y\in \bm u(\partial \Omega)} h\cdot y \right\}.
\end{equation*}
Hence
\begin{equation*}
    \bm u(\Omega)\subset \bigcap_{h\in\R^N, \abs{h}=1}\left\{x\in\R^N:x\cdot h\leq \max_{y\in \bm u(\partial \Omega)} y\cdot h \right\}=:M.
\end{equation*}
It remains to show that $M=\convhull \bm u(\partial \Omega)$. 

Clearly $\convhull \bm u(\partial\Omega)\subset M$, because $M$ is an intersection of convex sets, each containing $\bm u(\partial\Omega)$. 
On the other hand, $\convhull \bm u(\partial \Omega)$ is the intersection of all closed half-spaces containing $\bm u(\partial\Omega)$, which follows from the Hahn-Banach theorem, as we can separate any closed convex set $K\subset\R^N$ from a point $x_0\in\R^N\setminus K$ by a closed half-space $H$ (meaning that $K\subset H$ and $x_0\notin H$). Every closed half-space is of the form \[H_{h,\ell}:=\{x\in\R^N:x\cdot h \leq \ell\}\] for some $h\in\R^N$, $\abs h = 1$ and $\ell\in\R$. If $\ell<\max\limits_{y\in \bm u(\partial \Omega)}y\cdot h$, then $H_{h,\ell}\not\subset \bm u(\partial\Omega) $. So $M\subset \convhull \bm u(\partial\Omega)$ and the proof is finished.
\end{proof}

Note that the proof uses heavily the symmetry of the Laplace operator and thus cannot be easily adopted to other problems.

\subsection{Scalar case and motivation of the method}

To motivate our method, we hereby redo the classic proof of weak maximum principle for weak solutions of linear elliptic equations.
We state and prove it in such a way as to give a guide in how proceed for the convex hull property in the case of systems of equations. For simplicity, we omit the lower order terms.

\begin{theorem}[Maximum principle for linear second-order elliptic PDE]\label{thm:MPweaksol2}
Consider the problem 
\begin{equation*}
     -\div(A\nabla u)=0 \quad \text{in }\Omega,
\end{equation*}
where $A\in L^\infty (\Omega; \R^{n\times n})$ is uniformly elliptic, that is, for some $\lambda>0$ and a.e. $x\in\Omega$
\begin{equation*}
   A(x)\xi \cdot\xi\geq \lambda\abs{\xi}^2, \quad \xi\in\R^n.
\end{equation*}
Let $u\in W^{1,2}(\Omega;\R^N)$ be a weak solution of this problem, that is,
\begin{equation*}
    \int_\Omega A \nabla u \cdot \nabla \varphi \dx=0, \quad \varphi\in W^{1,2}_0(\Omega).
\end{equation*}
Denote
\begin{equation*}
    M=\esssup_{\partial\Omega} u ,\quad m=\essinf_{\partial\Omega} u.
\end{equation*}
Then $u\in [m,M]$ a.e. in $\Omega$.
\end{theorem}
\begin{proof}
We put $w= (u-M)^+ - (u-m)^-$, then $w\in W^{1,2}(\Omega)$ because positive and negative part are 1-Lipschitz functions. Moreover, $w$ has zero trace since $m\leq u\leq M$ a.e. on $\partial\Omega$, so even $w\in W^{1,2}_0(\Omega)$ and we can use it as a test function. So we obtain, since $\nabla u =\nabla (u-M)^+$  on $\supp (u-M)^+$ and 
$\nabla u =-\nabla (u-m)^-$  on $\supp (u-m)^-$: 
\begin{multline*}
    0=\int_\Omega A\nabla u \cdot \nabla w\dx =\int_\Omega A \nabla(u-M)^+ \cdot \nabla(u-M)^+ \dx +\int_\Omega A \nabla (u-m)^-\cdot \nabla(u-m)^- \dx 
    \\\geq \lambda\norm{\nabla(u-M)^+}_2^2 + \lambda\norm{\nabla(u-m)^-}_2^2 \geq \tilde\lambda\norm{(u-M)^+}_{W^{1,2}}^2 + \tilde\lambda\norm{(u-m)^-}_{W^{1,2}}^2.
\end{multline*}
Hence
\begin{equation*}
    (u-M)^+=0 \quad\text{and}\quad (u-m)^-=0 \quad\text{a.e. in }\Omega,
\end{equation*}
which simply means that $u\in[m,M]$ a.e. in $\Omega$.
\end{proof}

Let us now comment on how the approach in this proof can be generalised to systems of equations. If we denote $K=[m,M]$, then $K$ is the closed convex hull of the (essential) boundary values of $u$. Let $\Pi_K\colon \R\to K$ be the projection
to $K$, that is, 
\begin{equation*}
    \Pi_K(x)=\argmin_{y\in K}\abs{y-x}= x-(x-M)^+ +(x-m)^-.
\end{equation*}
If $w= (u-M)^+ - (u-m)^-$ as in the proof above, we have $\Pi_K u = u-w$. Hence the test function used in the proof is $w=u-\Pi_K u$. So the idea when $\bm u$ takes values in $\R^N$ can be analogous: take $K$ to be the closed convex hull of the boundary values of $\bm u$, and if $\Pi_K\colon \R^N\to K$ is the projection of $\R^N$ to $K$, use $\bm u-\Pi_K \bm u$ as the test function and try to prove that $\bm u=\Pi_K \bm u$ in $W^{1,2}(\Omega;\R^N)$. 

But there are some caveats. In the proof of Theorem \ref{thm:MPweaksol2} we used that $\nabla u = \nabla (u-M)^+ -\nabla(u-m)^-$, provided $u\notin [m,M]$. However, in $\R^N$, it is not true that $\nabla \bm u=\nabla (\bm u-\Pi_K \bm u)$ if $\bm u\notin K$. It only holds that (as we will see in Section \ref{chap:proj})
\begin{equation}\label{eqn:motivgradineq}
    \nabla \bm u : \nabla\piku \geq \abs{\nabla\Pi_K \bm u}^2.
\end{equation}  We can try to prove the convex hull property only using this inequality instead, but we may run into problems. For instance, if the highest-order term in the equation is $-\div(\mathbb A\nabla \bm u)$, then the corresponding term after using the test function $\bm u-\piku$ is $\int_\Omega \mathbb A\nabla \bm u : \nabla(\bm u-\piku)\dx$. First, we would want $\nabla \piku$ in place of $\nabla \bm u$, but if $\mathbb A$ is linear and elliptic, we may write it as
\begin{equation*}
    \int_\Omega \mathbb A\nabla\bm u : \nabla(\bm u-\piku)\dx=\int_\Omega \mathbb A\nabla (\bm u-\piku) : \nabla(\bm u-\piku)\dx+\int_\Omega \mathbb A\nabla \piku : \nabla(\bm u-\piku)\dx
\end{equation*}
and estimate the first term by ellipticity of $\mathbb A$. 
But the second term cannot be directly estimated by (\ref{eqn:motivgradineq}) because of the coefficients $\mathbb A$. However, if the modify the projection $\Pi_K$ so that it will instead yield the inequality 
\begin{equation*}
    \mathbb A\nabla\piku : \nabla (\bm u-\piku)\geq 0,
\end{equation*}
we could estimate the second term. A plausible idea would be to take a different inner product on $\R^N$ and make the projection with respect to it. But we cannot take ``inner product with respect to $\mathbb A$'' because it even formally doesn't make sense -- $\mathbb A$ is a fourth-order tensor in $\R^{(N\times n)\times (N\times n)}$, not an $N\times N$ matrix. This indicates the need for the decoupling of the coefficients as in Theorem \ref{thm:CHPellipticlin}. That this condition cannot be relaxed is demonstrated by Example \ref{ex:counterexell}.  

\section{Projections to convex sets}\label{chap:proj}

In this section we develop the theory for projections of Sobolev functions to closed convex sets. First, we review some rather well-known properties of pointwise projections. Second, we apply this to Sobolev functions, using the approximation by difference quotients. The crucial result to be used in the following proofs of the convex hull property is the inequality in Proposition \ref{prop:gradineq}.

\subsection{Properties of projections}

We review some well-known properties of projections to closed convex sets in\footnote{Note that the same results hold in a Hilbert space.} $\R^N$.

\begin{definition}
Let $A\in \R^{N\times N}$ be a symmetric positive definite matrix, denote by $\scalar{v}{w}_A:= v\cdot Aw$ the corresponding inner product on $\R^N$ and $\abs v _A:=\sqrt {\scalar{v}{v}_A}$ the induced norm. Let $K\subset \R^N$ be a nonempty closed convex set. We define the projection of $\R^N$ to $K$, denoted by $\Pi^A_K\colon \R^N \to K$, by setting for $x\in \R^N$
\begin{equation*}
    \Pi^A_K x=\argmin_{y\in K} \abs{x-y}_A.
\end{equation*}
It is straightforward to check that $\Pi^A_K$ is uniquely defined. If $A=I$ the identity matrix and thus we have the Euclidean inner product on $\R^N$, we write $\Pi_K^I=:\Pi_K$.
\end{definition}

\begin{proposition}\label{prop:convexprop}
Let $K\subset \R^N$ be a nonempty closed convex set. Then for all $x\in \R^N$ and all $z\in K$
\begin{equation*}\label{eqn:productprop}
\scalar{x-\Pi^A_K x}{z -\Pi^A_K x}_A\leq 0.
\end{equation*} 
\end{proposition}
\begin{proof}
Let $x\in \R^N$ and $z\in K$. We have $\Pi^A_K x = \argmin_{y\in K} \scalar{x-y}{x-y}_A$. Thus for any $h\in(0,1)$ we obtain (since $(1-h)\Pi^A_K x+ hz\in K$ from convexity of $K$):
\begin{multline*}
    \scalar{x-\Pi^A_K x}{x-\Pi^A_K x}_A \leq \scalar{x-((1-h)\Pi^A_K x+ hz)}{x-((1-h)\Pi^A_K x+ hz)}_A \\=\scalar{(1-h)(x-\Pi^A_K x)+h(x-z)}{(1-h)(x-\Pi^A_K x)+h(x-z)}_A\\=(1-h)^2\scalar{x-\Pi^A_K x}{x-\Pi^A_K x}_A+h^2\scalar{x-z}{x-z}_A+2h(1-h)\scalar{x-\Pi^A_K x}{x-z}_A.
\end{multline*}
So 
\begin{equation*}
    (2h-h^2)\scalar{x-\Pi^A_K x}{x-\Pi^A_K x}_A \leq h^2\scalar{x-z}{x-z}+2h(1-h)\scalar{x-\Pi^A_K x}{ x-z}_A.
\end{equation*}
Dividing by $h$ gives
\begin{equation*}
    (2-h)\scalar{x-\Pi^A_K x}{x-\Pi^A_K x}_A \leq h\scalar{x-z}{x-z}_A+2(1-h)\scalar{x-\Pi^A_K x}{ x-z}_A.
\end{equation*}
Passing to the limit as $h\to 0^+$ and rearranging the terms yields the result.
\end{proof}

\begin{corollary}\label{cor:projineq}
Let $K\subset \R^N$ be a nonempty convex closed set. Then for all $x,y\in \R^N 
$
\begin{equation*}
    \scalar{x-y}{\Pi^A_K x -\Pi^A_K y}_A\geq \abs{\Pi^A_K x -\Pi^A_K y}_A^2.
\end{equation*}
\end{corollary}
\begin{proof}
From Proposition \ref{prop:convexprop}, putting $z:=\Pi^A_K y \in K$, we obtain the inequality
\begin{equation*}
    \scalar{x-\Pi^A_K x}{\Pi^A_K x -\Pi^A_K y}_A\geq 0.
\end{equation*}
Switching the roles of $x$ and $y$, we also have
\begin{equation*}
    \scalar{y-\Pi^A_K y}{\Pi^A_K y -\Pi^A_K x}_A\geq 0.
\end{equation*}
Subtracting these two inequalities and adding $\abs{\Pi^A_K x -\Pi^A_K y}_A^2$ yields the result.
\end{proof}

\begin{corollary}\label{cor:projlipsch}
Let $K\subset \R^N$ be a nonempty convex closed set. Then $\Pi^A_K\colon \R^N\to \R^N$ is 1-Lipschitz.
\end{corollary}
\begin{proof} Immediate from Corollary \ref{cor:projineq} and the Cauchy-Schwarz inequality.
\end{proof}

\subsection{Projections of Sobolev functions}
In this section we study projections of $\R^N$-valued Sobolev functions to closed convex sets. For this we recall a standard result on the convergence of difference quotients.

\begin{definition}[Difference quotients]
For a function $\bm u\colon \Omega\to \R^N$ we define
\begin{equation*}
    \Delta^h_i \bm u(x)=\frac{\bm u(x+he_i)-\bm u(x)}{h},
\end{equation*}
for all $0\neq h\in \R$, $i\in\{1,\dots,n\}$, and $x\in\Omega$ such that $x+he_i\in\Omega$, where $e_i\in \R^n$ is the $i$-th canonical basis vector.
\end{definition}

\begin{proposition}[Difference quotients and weak partial derivative]\label{prop:differencequotients}
Let $\bm u\in W^{1,p}(\Omega;\R^N)$, $1\leq p < \infty$. Then it holds that
\begin{equation*}
    \Delta^h_i \bm u \to \pd{\bm u}{x_i} \quad\text{in }L_{\operatorname{loc}}^p(\Omega; \R^N), \text{ as }h\to 0.
\end{equation*}
In particular, 
\begin{equation*}
    \Delta^h_i \bm u \to \pd{\bm u}{x_i} \quad\text{a.e. in }\Omega,\text{ as }h\to 0.
\end{equation*}
\end{proposition}
This result is standard and can be found in any textbook treating Sobolev spaces, e.g. \cite{gilbargEllipticPartialDifferential2001}.
Now let us formulate a version of Corollary \ref{cor:projineq} for projections of Sobolev functions. Henceforth, by writing $\Pi_K \bm u$ we mean the composition $\Pi_K \circ \bm u$.

\begin{proposition}\label{prop:gradineq}
Let $1\leq p<\infty$ , $\bm u\in W^{1,p}(\Omega;\R^N)$ and let $K\subset \R^N$ be nonempty, closed and convex. Then 
\begin{equation}\label{eqn:gradineq}
    \nabla \bm u:\nabla \Pi_K \bm u\geq \abs{\nabla\Pi_K \bm u}^2 \quad \text{a.e. in }\Omega.
\end{equation}
\end{proposition}
\begin{remark}
We will often be using the inequality (\ref{eqn:gradineq}) in the form
\begin{equation*}
    \nabla (\bm u-\piku):\nabla \piku \geq 0.
\end{equation*}
\end{remark}
\begin{proof}
Let $x\in\Omega$ and choose $i\in\{1,\dots,n\}$. Find $\delta>0$ such that for every $|h|<\delta$ we have $x+he_i\in\Omega$. Then for such $h$, Corollary \ref{cor:projineq} gives 
\begin{equation*}
    (\bm u(x+he_i)-\bm u(x))\cdot (\Pi_K \bm u(x+he_i)-\Pi_K \bm u(x))\geq \abs{\Pi_K \bm u(x+he_i)-\Pi_K \bm u(x)}^2.
\end{equation*}
After dividing by $h^2$, this reads with the difference quotient notation as
\begin{equation*}
    \Delta^h_i \bm u(x)\cdot \Delta^h_i \Pi_K \bm u(x)\geq \abs{\Delta^h_i \Pi_K \bm u(x)}^2.
\end{equation*}
By Corollary \ref{cor:projlipsch}, $\Pi_K \bm u\in W^{1,p}(\Omega;\R^N)$. Hence by Proposition \ref{prop:differencequotients},
 $\Delta^h_i \bm u \to \pd{\bm u}{x_i}$ and $\Delta^h_i \Pi_K \bm u\to \pd{\Pi_K \bm u}{x_i}$ a.e. in $\Omega$, as $h\to 0$.
Thus passing to the limit as $h\to 0$ gives for a.e. $x\in\Omega$ 
\begin{equation*}
    \pd{\bm u}{x_i}(x)\cdot \pd{\Pi_K \bm u}{x_i}(x)\geq \abs{\pd{\Pi_K \bm u}{x_i}(x)}^2.
\end{equation*} 
Summing over $i=1,\dots,n$ gives 
\begin{equation*}
    \sum_{i=1}^n \pd{\bm u}{x_i}(x)\cdot \pd{\Pi_K \bm u}{x_i} (x)\geq \sum_{i=1}^n \abs{\pd{\Pi_K \bm u}{x_i}(x)}^2
\end{equation*}
for a.e. $x\in\Omega$, which is exactly our claim.
\end{proof}

\begin{proposition}\label{prop:gradineqA}
Let $\bm u\in W^{1,p}(\Omega;\R^N)$, $1\leq p<\infty$ and let $A\in \R^{N\times N}$ be a positive definite matrix. Let $K\subset \R^N$ be nonempty, and closed convex. Then for a.e. $x\in \Omega$ it holds
\begin{equation*}
\scalar{\pd{\bm u}{x_i}(x)}{\pd{\Pi_K^A \bm u}{x_i}(x)}_A\geq \abs{\pd{\Pi_K^A \bm u}{x_i}(x)}_A^2,\quad i=1,\dots,n.
\end{equation*}
\end{proposition}
\begin{proof}
We proceed as in the proof of Proposition \ref{prop:gradineq}, replacing the Euclidean inner product $\cdot$ with $\scalar{\,}{\,}_A$ and the projection $\Pi_K$ with $\Pi_K^A$, and omitting the last summation over $i$.
\end{proof}

\begin{proposition}\label{prop:projorth}
Let $1\leq p<\infty$ , $\bm u\in W^{1,p}(\Omega;\R^N)$ and let $K\subset \R^N$ be nonempty, closed and convex. Then for $i=1,\dots,n$ it holds
\begin{equation}\label{eqn:projorth}
    \pd{\piku}{x_i}\cdot(\bm u-\piku)=0 \quad \text{a.e. in }\Omega.
\end{equation}
\end{proposition}
\begin{proof}
  We proceed similarly as before. By applying twice Proposition \ref{prop:convexprop} we have (for $x\in\Omega$ and $h>0$ small enough so that $x\pm he_i\in\Omega$)
    \begin{align*}
    ( \Pi_K \bm u(x) -
    \Pi_K \bm u(x-he_i) )\cdot (\bm u(x) - \Pi_K \bm u(x))  &\geq 0,
    \\
    ( \Pi_K \bm u(x) -
    \Pi_K \bm u(x+he_i) ) \cdot (\bm u(x) - \Pi_K \bm u(x)) &\geq 0.
  \end{align*}
  Dividing by $h$ and passing to the limit $h\to 0$, Proposition \ref{prop:differencequotients} gives for a.e. $x\in\Omega$
    \begin{align*}
    \pd{\piku}{x_i}(x)\cdot (\bm u(x) - \Pi_K \bm u(x))  &\geq 0,
    \\
    -\pd{\piku}{x_i}(x) \cdot (\bm u(x) - \Pi_K \bm u(x)) &\geq 0.
  \end{align*}
  Hence the result.
\end{proof}

Now, since we will be concerned with parabolic equations, we need also to treat the time-derivative term. The lack of enough time regularity to perform the difference-quotient approach can be overcome by mollification in time, as will be demonstrated below.
\begin{lemma}
  Let $\bm u\in C^1([0,T]; W^{1,p}(\Omega; \R^N))$ with $1<p<\infty$ and let $K\subset\R^N$ be nonempty, closed and convex. Then  $ \Pi_K \bm u \in W^{1,\infty}((0,T);W^{1,p}(\Omega; \R^N))$ and for a.e. $t\in(0,T)$ we have 
  \begin{equation*}
    (\bm u(t) - \Pi_K \bm u(t)) \cdot \partial_t \Pi_K \bm u(t) = 0 \quad \text{ a.e. in } \Omega.
  \end{equation*}
\end{lemma}
\begin{proof}
Using that $\Pi_K$ is 1-Lipschitz (Corollary \ref{cor:projlipsch}) and the chain rule we get the estimate 
\begin{equation*}
    \abs{\nabla (\piku(t,x))}= \abs{(\nabla \Pi_K)(\bm u(t,x))\nabla \bm u(x,t)}\leq  \abs{(\nabla \Pi_K)(\bm u(t,x))}\abs{\nabla \bm u(x,t)}\leq \abs{\nabla \bm u(t,x)},
\end{equation*}
so using this we get for $t,s\in[0,T]$
\begin{multline*}
    \norm{\piku(t)-\piku(s)}_{W^{1,p}}^p = \norm{\piku (t)-\piku(s)}_p^p +\norm{\nabla(\piku (t))-\nabla(\piku(s))}_p^p \\
    \leq  \norm{\bm u (t)-\bm u(s)}_p^p +\norm{\nabla(\bm u (t))-\nabla(\bm u(s))}_p^p = \norm{\bm u(t)-\bm u(s)}_{W^{1,p}}^p\leq \norm{\partial_t \bm u}_{C([0,T];W^{1,p}(\Omega;\R^N))}^p(t-s)^p.
\end{multline*}
Therefore $\piku$ is Lipschitz in time (treated as a function $\piku \colon (0,T)\to W^{1,p}(\Omega;\R^N)$), so that by Rademacher theorem\footnote{Here we use that $W^{1,p}(\Omega;\R^N)$ is reflexive.} $\partial_t \Pi_K \bm u(t)\in  W^{1,p}(\Omega; \R^N)$ exists for a.e. $t\in(0,T)$ and it is also the weak time derivative. Thus for such $t$ we have a sequence $h_j\to 0^+$ such that for $j\to \infty$ and a.e. $x\in \Omega$
\begin{align*}
\frac{\piku(t+h_j,x)-\piku(t,x)}{h_j} \to \partial_t\piku (t,x), \\
\frac{\piku(t-h_j,x)-\piku(t,x)}{-h_j} \to \partial_t\piku (t,x).
\end{align*}
Further, we have by Proposition \ref{prop:convexprop} for a.e. $x\in\Omega$
  \begin{align*}
    \big(\bm u(t,x) - \Pi_K \bm u(t,x)\big) \cdot \big( \Pi_K \bm u(t,x) - \Pi_K \bm u(t-h_j,x) \big) &\geq 0,
    \\
    \big(\bm u(t,x) - \Pi_K \bm u(t,x)\big) \cdot \big( \Pi_K \bm u(t,x) - \Pi_K \bm u(t+h_j,x) \big) &\geq 0.
  \end{align*}
Dividing by $h_j$ and passing to the limit $j\to \infty$ we get
  \begin{align*}
    \big(\bm u(t,x) - \Pi_K \bm u(t,x)\big) \cdot \partial_t \Pi_K \bm u(t,x)
    &\geq 0,
    \\
    -\big(\bm u(t,x) - \Pi_K \bm u(t,x)\big) \cdot \partial_t \Pi_K \bm u(t,x)
    &\geq 0.
  \end{align*}
  This proves the claim.
\end{proof}
\begin{proposition}\label{prop:timederiv}
  Let $\bm u\in L^p((0,T); W^{1,p}(\Omega;\R^N))\cap C([0,T];L^2(\Omega; \R^N))$ with \\ $\partial_t \bm u \in L^{p'}((0,T);W^{-1,p'}(\Omega;\R^N))$ and let $K\subset \R^N$ be a closed convex set satisfying $\bm u(t,x)\in K$ for a.e. $(t,x)\in(0,T)\times \partial\Omega$. Then for $0\leq t_1<t_2\leq T$
  \begin{equation*}
   \int_{t_1}^{t_2}\scalar{\partial_t \bm u(t)}{\bm u(t) - \Pi_K \bm u(t)}\dt =\int_\Omega \frac{\abs{\bm u(t_2)-\Pi_K\bm u(t_2)}^2}{2}\dx-\int_\Omega \frac{\abs{\bm u(t_1)-\Pi_K\bm u(t_1)}^2}{2}\dx.
  \end{equation*}
  In particular, for a.e. $t\in (0,T)$ it holds 
  \begin{equation*}
    \scalar{\partial_t \bm u(t)}{\bm u(t) - \Pi_K \bm u(t)}=\partial_t \int_\Omega \frac{\abs{\bm u(t)-\Pi_K\bm u(t)}^2}{2}\dx.
  \end{equation*}
\end{proposition}
\begin{proof}
The strategy is to mollify in time and use the previous lemma. Let first $0<t_1<t_2<T$. Set $\bm u_\delta = \bm u* \psi_\delta$ where $*$ is convolution in time, i.e. $\bm u_\delta(t,x)=\int_{t-\delta}^{t+\delta} \bm u(s,x)\psi_\delta(t-s)\dx[s]$, where $\psi_\delta\in C^\infty_c(\R)$ is the standard mollification kernel with radius $\delta$, and $\delta <\min(t_1,T-t_2)$. Then by a standard result on mollification we have $\bm u_\delta\in C^\infty([t_1,t_2]; W^{1,p}(\Omega;\R^N))$ and $\bm u_\delta\in C^\infty([t_1,t_2]; L^2(\Omega;\R^N))$. So using the previous lemma, we find that
  \begin{equation*}
      \int_{t_1}^{t_2}\int_\Omega\partial_t \bm u_\delta \cdot (\bm u_\delta - \piku_\delta)\dx\dt = \int_{t_1}^{t_2}\int_\Omega(\partial_t \bm u_\delta - \partial_t \piku_\delta) \cdot (\bm u_\delta - \piku_\delta)\dx\dt.
  \end{equation*}
  Further, as $\bm u_\delta - \piku_\delta \in W^{1,\infty}((0,T);L^2(\Omega;\R^n))$, which can be shown analogously as in the previous lemma, we obtain
\begin{equation*}
 \int_{t_1}^{t_2}\int_\Omega(\partial_t \bm u_\delta - \partial_t \piku_\delta) \cdot (\bm u_\delta - \piku_\delta)\dx\dt = \int_{t_1}^{t_2}\frac 12 \partial_t \int_\Omega \abs{\bm u_\delta - \piku_\delta}^2\dx\dt,
  \end{equation*}
and finally  \begin{equation*}
      \int_{t_1}^{t_2}\frac 12 \partial_t \int_\Omega \abs{\bm u_\delta - \piku_\delta}^2\dx\dt = \int_\Omega \frac{\abs{\bm u_\delta(t_2)-\Pi_K\bm u_\delta(t_2)}^2}{2}\dx-\int_\Omega \frac{\abs{\bm u_\delta(t_1)-\Pi_K\bm u_\delta(t_1)}^2}{2}\dx.
  \end{equation*}
We have thus shown that the desired equality holds for $\bm u_\delta$, so now it remains to pass to the limit $\delta \to 0$.

Note first that by the assumption on $K$ it holds $\bm u(t)-\piku(t) \in W^{1,p}_0(\Omega;\R^N)$ for a.e. $t$. Then by standard results on mollification we get $\bm u_\delta \to \bm u$ in $L^p((0,T);W^{1,p}_0(\Omega;\R^N))$, and consequently by $\Pi_K$ being Lipschitz, also $\piku_\delta \to \piku$ in $L^p((0,T);W^{1,p}(\Omega;\R^N))$. By the same result on mollification we obtain $\partial_t \bm u_\delta \to \partial_t \bm u$ in $L^{p'}((0,T);W^{-1,p'}(\Omega;\R^N))$, since it holds $\partial_t \bm u_\delta = \partial_t (\bm u*\psi_\delta)= (\partial_t \bm u)*\psi_\delta$. Thus we can pass to the limit on the left hand side: 
\begin{equation*}
\int_{t_1}^{t_2}\int_\Omega\partial_t \bm u_\delta \cdot (\bm u_\delta - \piku_\delta)\dx\dt \to 
\int_{t_1}^{t_2}\langle \partial_t \bm u , \bm u - \piku\rangle \dt.
\end{equation*}

Further, by mollification we have $\bm u_\delta \to \bm u$ in $C([0,T];L^2(\Omega;\R^N))$, and also $\piku_\delta \to \piku$ in $C([0,T];L^2(\Omega;\R^N))$, as $\Pi_K$ is Lipschitz. This allows to pass to the limit on the right hand side: 
\begin{multline*}
\int_\Omega \frac{\abs{\bm u_\delta(t_2)-\Pi_K\bm u_\delta(t_2)}^2}{2}\dx-\int_\Omega \frac{\abs{\bm u_\delta(t_1)-\Pi_K\bm u_\delta(t_1)}^2}{2}\dx \\
\to 
\int_\Omega \frac{\abs{\bm u(t_2)-\Pi_K\bm u(t_2)}^2}{2}\dx-\int_\Omega \frac{\abs{\bm u(t_1)-\Pi_K\bm u(t_1)}^2}{2}\dx.
\end{multline*}

This concludes the proof if $0<t_1<t_2<T$. Passing to the limits $t_1\to 0$ or $t_2\to T$, using on the left Lebesgue dominated convergence theorem and on the right $\bm u,\piku \in C([0,T];L^2(\Omega;\R^N)$, we conclude the proof for all cases $0\leq t_1<t_2\leq T$.

Finally, we note that the quantity on the left is clearly absolutely continuous in time, so for a.e. $t\in (0,T)$ we can differentiate to obtain the second expression.
\end{proof}
\section{Convex hull property}\label{chap:CHP}

This chapter covers our two results on the convex hull property, namely for the elliptic and nonlinear parabolic systems.

\subsection{Linear elliptic systems}
In this section we prove a convex hull property for linear elliptic systems satisfying suitable condition on the coefficients. 
In particular we consider the system
\begin{equation}\label{eqn:ellsys}
	-\div (\mathbb A\nabla \bm u)=\bm 0\quad \text{in }\Omega,
\end{equation}
with coefficients $\mathbb A=(\mathbb A_{ij}^{\alpha\beta})_{i,j=1,\dots,n}^{\alpha,\beta=1,\dots,N} \in L^\infty (\Omega;\R^{(N\times n)\times (N\times n)})$, written in components as
\begin{equation*}
-\sum_{\beta=1}^N\sum_{i,j=1}^n\pd{}{x_i}\left( \mathbb A^{\alpha\beta}_{ij}(x) \pd{u^{\beta}}{x_j}(x) \right) = 0,\quad \text{a.e. }x\in\Omega,\ \alpha=1,\dots, N.
\end{equation*}

\begin{definition}[Weak solution]
By a \emph{weak solution} to \eqref{eqn:ellsys} we mean $\bm u\in W^{1,2}(\Omega; \R^N)$ which satisfies
\begin{equation}\label{eqn:elldiagweak}
   \int_\Omega \mathbb A\nabla \bm u  : \nabla \bm \varphi \dx =0
\end{equation}
for all $\bm \varphi\in W^{1,2}_0(\Omega;\R^N)$.
Written in components, \eqref{eqn:elldiagweak} reads as 
\begin{equation*}
    \int_\Omega  \sum_{\alpha,\beta=1}^N \sum_{i,j=1}^n \mathbb{A}^{\alpha\beta}_{ij}(x)\pd{u^{\beta}}{x_j}  (x) \pd{\varphi^{\alpha}}{x_i} (x)\dx =0.
\end{equation*}
\end{definition}

To formulate the convex hull property for weak solutions, we now precisely define the convex hull of boundary values for Sobolev functions.

\begin{definition}[Convex hull of boundary values]\label{def:CHPell}
 Let $\bm u\in W^{1,p}(\Omega; \R^N)$. Then we define \emph{closed convex hull of boundary values} $\convhull \bm u(\partial\Omega) \subset \R^N$ to be the smallest closed convex set such that $\bm u(x)\in \convhull \bm u(\partial\Omega)$ for a.e. $x\in \partial\Omega$.
Equivalently,
\begin{equation*}
\convhull \bm u(\partial\Omega) = \bigcap \left\{K\subset \R^N: K\text{ closed convex},\bm u(x)\in K \text{ for a.e. }x\in\partial \Omega \right\},
\end{equation*}
thus it is well-defined.
\end{definition}

\begin{remark}\label{rem:essbdval-ell}
Alternatively, one can define the $\convhull \bm u(\partial\Omega)$ as the closed convex hull of the essential boundary values of $\bm u$. A $y\in \R^N$ is an \emph{essential boundary value} of $\bm u$ if for all $\epsilon>0$, the preimage $\bm u|_{\partial\Omega}^{-1}(B_\epsilon(y))$ has positive surface measure, where $\bm u|_{\partial \Omega}$ denotes the trace of $\bm u$ on $\partial \Omega$, and $B_\epsilon(y)$ is the $\epsilon$-ball centred at $y$. It is not hard to check that this definition is equivalent to the one given above.
\end{remark}

\begin{theorem}[Convex hull property for linear elliptic systems]\label{thm:CHPellipticlin}

Suppose that $\mathbb A=(\mathbb A_{ij}^{\alpha\beta})_{i,j=1,\dots,n}^{\alpha,\beta=1,\dots,N} \in L^\infty (\Omega;\R^{(N\times n)\times (N\times n)})$ has the form 
\begin{equation*}
    \mathbb A_{ij}^{\alpha\beta}(x)=\mathcal{A}^{\alpha\beta} a_{ij}(x),
\end{equation*}
where $\mathcal{A}=(\mathcal{A}^{\alpha\beta})_{\alpha,\beta=1}^N \in\R^{N\times N}$ is symmetric positive definite and $\bm a=(a_{ij})_{i,j=1}^n\in L^\infty(\Omega;\R^{n\times n})$ is symmetric and uniformly positive definite, that is, there is $\lambda>0$ such that for a.e. $x\in\Omega$
\begin{equation}\label{eqn:aunifell}
    \forall \xi\in\R^n:\xi \cdot \bm a(x)\xi \geq \lambda \abs{\xi}^2.
\end{equation}
Then every weak solution $\bm u$ to \eqref{eqn:ellsys} satisfies the convex hull property $\bm u(x)\in \convhull \bm u(\partial\Omega)$ for a.e. $x\in\Omega$.
\end{theorem}
\begin{proof}
Put $K=\convhull \bm u(\partial\Omega)$, where this set is understood as in Definition \ref{def:CHPell} and take $\bm w = \bm u-\Pi_K^{\mathcal A} \bm u$. Now, since $\Pi_K^{\mathcal A}$ is Lipschitz by Corollary \ref{cor:projlipsch}, we have that $\bm w\in W^{1,2}(\Omega;\R^N)$. Moreover, because $\bm u=\Pi_K^{\mathcal{A}} \bm u$ on $\partial\Omega$, we have even $\bm w\in W^{1,2}_0(\Omega;\R^N)$. So we can use $\bm w$ as a test function in (\ref{eqn:elldiagweak}) and obtain 
\begin{equation}\label{eqn:testdiag}
	0=\int_\Omega \mathbb A\nabla \bm u : \nabla \bm w \dx=\int_\Omega \mathbb A\nabla (\bm u-\Pi_K^{\mathcal A} \bm u) : \nabla (\bm u-\Pi_K^{\mathcal A} \bm u) \dx+\int_\Omega \mathbb A\nabla \Pi_K^{\mathcal A} \bm u : \nabla (\bm u-\Pi_K^{\mathcal A} \bm u)\dx.
\end{equation}

Now let us for a.e. $x\in\Omega$ diagonalize the positive definite matrix $\bm a(x)$, so that 
\begin{equation*}
    \bm a(x)= Q^T (x)D(x)Q(x),
\end{equation*}
where $Q(x)\in\R^{n\times n} $ is an orthogonal matrix and $D(x)\in\R^{n\times n}$ a diagonal matrix with diagonal entries $d_i(x)$, $i=1,\dots,n$. Moreover, because of (\ref{eqn:aunifell}), $d_i\in L^\infty (\Omega)$ satisfy
\begin{equation}\label{eqn:dialpha}
    d_i(x)\geq \lambda>0.
\end{equation}

Now fix some $x_0\in\Omega$. Denote $\overline{Q}=Q(x_0)=(\overline q_{ij})_{i,j=1}^n$ and put $\overline{\bm u}(y)=\bm u(\overline Q^T y)$, so that $\bm u(x)=\overline{\bm u}(\overline Qx)$. Then, by the chain rule, with $y=\overline Q x$
\begin{equation}\label{eqn:partialy}
    \pd{\overline{u}^\alpha}{y_i}(y)=\pd{}{y_i}(u^\alpha(\overline{Q}^T y))=\sum_{j=1}^n \pd{u^\alpha}{x_j}(\overline{Q}^T y) \pd{}{y_i}\sum_{k=1}^n \overline{q}_{k j}y_k =\sum_{j=1}^n \overline{q}_{ij}\pd{u^\alpha}{x_j}(x),\quad \text{i.e. } \nabla_y \overline{\bm u}(y)= \nabla \bm u (x) \overline Q^T.
\end{equation}

Further, from positive definiteness of ${\mathcal A}$ we get $\gamma>0$ such that 
\begin{equation}\label{eqn:Agamma}
    {\mathcal A}\xi\cdot \xi \geq \gamma \abs{\xi}^2, \quad \xi \in \R^N .
\end{equation}

Now we consider the first term on the right hand side of (\ref{eqn:testdiag}) at the point $x_0$. Denote $y_0=\overline{Q}x_0$ and compute using (\ref{eqn:partialy}) (recall also that $a_{ij}(x_0)=\sum_{k=1}^n \overline{q}_{ki} d_k(x_0)\overline{q}_{kj}$)
\begin{multline}\label{eqn:computeAQDQ}
    \mathbb A(x_0)\nabla (\bm u(x_0)-\Pi_K^{\mathcal A} \bm u(x_0)) : \nabla (\bm u(x_0)-\Pi_K^{\mathcal A} \bm u(x_0))\\=
    \sum_{\alpha,\beta=1}^N \sum_{i,j,k=1}^n \mathcal{A}^{\alpha\beta} \overline{q}_{ki} d_k(x_0) \overline{q}_{kj} \pd{(\bm u - \Pi_K^{\mathcal A} \bm u)^\alpha}{x_i}(x_0)\pd{(\bm u - \Pi_K^{\mathcal A} \bm u)^\beta}{x_j}(x_0)\\=
    \sum_{k=1}^n  d_k(x_0) \mathcal{A} \sum_{i=1}^n \overline{q}_{ki} \pd{(\bm u - \Pi_K^{\mathcal A} \bm u)}{x_i}(x_0) \cdot\sum_{j=1}^n \overline{q}_{kj}\pd{(\bm u - \Pi_K^{\mathcal A} \bm u)}{x_j}(x_0)\\=
    \sum_{k=1}^n  d_k(x_0) \mathcal{A} \pd{(\overline{\bm u} - \Pi_K^{\mathcal A} \overline{\bm u})}{y_k}(y_0) \cdot \pd{(\overline{\bm u} - \Pi_K^{\mathcal A} \overline{\bm u})}{y_k}(y_0).
\end{multline}
Hence, using (\ref{eqn:Agamma}) and (\ref{eqn:dialpha}) we continue the calculation by
\begin{equation*}
    \sum_{k=1}^n  d_k(x_0) \mathcal{A} \pd{(\overline{\bm u} - \Pi_K^{\mathcal A} \overline{\bm u})}{y_k}(y_0) \cdot \pd{(\overline{\bm u} - \Pi_K^{\mathcal A} \overline{\bm u})}{y_k}(y_0)\geq \sum_{k=1}^n\lambda\gamma\abs{\pd{(\overline{\bm u} - \Pi_K^{\mathcal A} \overline{\bm u})}{y_k}(y_0)}^2\!\!=\lambda\gamma \abs{\nabla_y(\overline{\bm u}- \Pi_K^{\mathcal A}\overline{\bm u})(y_0)}^2
\end{equation*}
Using (\ref{eqn:partialy}) and recalling that $\overline{Q}$ is orthogonal, we get 
\begin{equation*}
 \abs{\nabla_y(\overline{\bm u}- \Pi_K^{\mathcal A}\overline{\bm u})(y_0)}=\abs{\nabla (\bm u -\Pi_K^{\mathcal A}\bm u)(x_0)\overline{Q}^T}=\abs{\nabla (\bm u -\Pi_K^{\mathcal A}\bm u)(x_0)}.
\end{equation*}
Hence 
\begin{equation*}
    \mathbb A(x_0)\nabla (\bm u(x_0)-\Pi_K^{\mathcal A} \bm u(x_0)) : \nabla (\bm u(x_0)-\Pi_K^{\mathcal A} \bm u(x_0))\geq \lambda\gamma\abs{\nabla (\bm u -\Pi_K^{\mathcal A}\bm u)(x_0)}^2.
\end{equation*}
Since $x_0\in\Omega$ was arbitrary, we can integrate this inequality over $\Omega$ to obtain 
\begin{equation}\label{eqn:ellestim}
        \int_\Omega \mathbb A\nabla (\bm u-\Pi_K^{\mathcal A} \bm u) : \nabla (\bm u-\Pi_K^{\mathcal A} \bm u)\dx\geq \lambda\gamma \norm{\nabla (\bm u -\Pi_K^{\mathcal A}\bm u)}_2^2\geq \tilde\gamma\norm{\bm u -\Pi_K^{\mathcal A}\bm u}_{W^{1,2}}^2,
\end{equation}
for some $\tilde \gamma>0$ from Poincar\'e inequality, since $\bm u -\Pi_K^{\mathcal A}\bm u$ has zero trace.

Regarding the  second term in (\ref{eqn:testdiag}), for fixed $x_0\in\Omega$, by analogous computation as \eqref{eqn:computeAQDQ} we get 
\begin{equation*}
    \mathbb A(x_0)\nabla \Pi_K^{\mathcal A} \bm u(x_0) : \nabla (\bm u(x_0)-\Pi_K^{\mathcal A} \bm u(x_0))=
    \sum_{k=1}^n  d_k(x_0) \mathcal{A} \pd{\Pi_K^{\mathcal A} \overline{\bm u}}{y_k}(y_0) \cdot \pd{(\overline{\bm u} - \Pi_K^{\mathcal A} \overline{\bm u})}{y_k}(y_0).
\end{equation*}
Using now on the right hand side Proposition \ref{prop:gradineqA} for $\overline{\bm u}$ and $\mathcal A$ gives us
\begin{equation*}
     \mathbb A(x_0)\nabla \Pi_K^{\mathcal A} \bm u(x_0) : \nabla (\bm u(x_0)-\Pi_K^{\mathcal A} \bm u(x_0))\geq 0.
\end{equation*}
Again, since $x_0\in\Omega$ was arbitrary, we can integrate this inequality over $\Omega$ to obtain
\begin{equation*}
     \int_\Omega \mathbb A \nabla \Pi_K^{\mathcal A} \bm u : \nabla (\bm u-\Pi_K^{\mathcal A} \bm u)\dx\geq 0.
\end{equation*}
Together with (\ref{eqn:ellestim}), looking back at (\ref{eqn:testdiag}) we see that $\norm{\bm u-\Pi_K^{\mathcal A} \bm u}_{W^{1,2}}\leq 0$, so $\bm u=\Pi_K^{\mathcal A} \bm u$ a.e. in $\Omega$ and thus $\bm u(x)\in K$  for a.e. $x\in\Omega$. This concludes the proof.
\end{proof}

\subsection{Nonlinear parabolic systems}

Now we consider the convex hull property for parabolic equations.
Throughout this section, let $T>0$ and put $Q_T=(0,T]\times \Omega$.

\begin{definition}[Parabolic boundary]
 We denote by $\parbd Q_T = (\{0\} \times \Omega) \cup ([0,T] \times \partial \Omega) $ the \emph{parabolic boundary of $Q_T$}.
\end{definition}
In the following we consider the Carath\'eodory functions (i.e. measurable in $(0,T)\times \Omega$ and continuous in the other variables)
\begin{align*}
 a_0&\colon (0,T)\times \Omega \times \R^N \times \R^{N\times n}\to [0,\infty), \\ 
\bm a&\colon (0,T)\times \Omega \to \R^{n\times n}, \\
    \bm b&\colon(0,T)\times\Omega\times\R^N\times\R^{N\times n} \to \R^n, \\
    c&\colon(0,T)\times\Omega\times\R^N\times\R^{N\times n}\to [0,\infty).
\end{align*}

Now we consider weak solutions to 
\begin{equation}\label{eqn:parnonlin}
    \partial_t \bm u - \div (\mathbb A\nabla \bm u)+\bm b \cdot \nabla \bm u + c\bm u=\bm 0 \quad\text{ in }Q_T,
\end{equation}
where $$\mathbb A_{ij}^{\alpha\beta}(t,x,\bm u,\nabla \bm u)=\begin{cases}a_0(t,x,\bm u,\nabla \bm u) a_{ij}(t,x), &\alpha=\beta\\ 0, &\alpha\neq \beta,
\end{cases} \text{ for }i,j=1,\dots,n;\ \alpha,\beta =1,\dots,N.$$ Written thus in components, the equation reads as the system
\begin{multline*}
 \partial_t u^\alpha(t,x) - \sum_{i=1}^n \pd{}{x_i}\left( a_0(t,x,\bm u,\nabla \bm u) \sum_{j=1}^n a_{ij}(t,x) \pd{u^\alpha}{x_j}(t,x)\right) \\+ \sum_{i=1}^n b_i(t,x,\bm u,\nabla \bm u) \pd{u^\alpha}{x_i}(t,x) + c(t,x,\bm u,\nabla \bm u) u^\alpha(t,x) = 0, \quad \alpha=1,\dots, N.
\end{multline*}
\begin{remark}
In the case that $a_0(t,x,\bm u,\nabla \bm u)=\abs{\nabla \bm u}^{p-2}$, $\bm a=I\in \R^{n\times n}$, $\bm b=\bm 0$, $c=0$, we obtain the parabolic $p$-Laplace equation 
\begin{equation*}
    \partial_t \bm u -\laplace_p \bm u:=\partial_t \bm u -\div(\abs{\nabla \bm u}^{p-2}\nabla \bm u)=\bm 0.
\end{equation*}
\end{remark}

\begin{definition}[Weak solution]
We call $\bm u\in L^p((0,T); W^{1,p}(\Omega; \R^N))\cap C([0,T];L^2(\Omega;\R^N))$ with  $1<p<\infty$ and $\partial_t \bm u\in L^{p'}((0,T); W^{-1,p'}(\Omega;\R^N))$ a \emph{weak solution} to the system (\ref{eqn:parnonlin}) if
\begin{equation*}
    \partial_t \bm u(t) - \div (\mathbb A(t)\nabla \bm u(t))+\bm b(t) \cdot \nabla \bm u(t) + c(t)\bm u(t)=\bm 0 \quad\text{ in }W^{-1,p'}(\Omega;\R^N) + L^2(\Omega;\R^N) 
\end{equation*}
for a.e. $t\in (0,T)$. That is, it holds
\begin{multline}\label{eqn:wspara}
\left\langle \partial_t \bm u(t), \bm \varphi \right\rangle +\int_\Omega a_0(t,x,\bm u(t,x),\nabla\bm u(t,x)) \sum_{\alpha=1}^N \sum_{i,j=1}^n a_{ij}(t,x)\pd{ u^\alpha}{ x_j}(t,x) \pd{\varphi^\alpha}{ x_i}(t,x) \\ + \!\sum_{\alpha=1}^N \!\sum_{i=1}^n b_i(t,x,\!\bm u(t,x),\nabla\bm u(t,x))\pd{ u^\alpha}{ x_i}(t,x) \varphi^\alpha(t,x) + c(t,x,\bm u(t,x), \!\nabla \bm u(t,x))\sum_{\alpha=1}^N \!u^\alpha(t,x)\varphi^\alpha(t,x) \dx=0
\end{multline}
for a.e. $t\in (0,T)$ and all $\bm \varphi\in W_0^{1,p}(\Omega;\R^N)\cap L^2(\Omega;\R^N)$.
\end{definition}

We also need to properly define the convex hull of (parabolic) boundary values $\bm u(\parbd Q_T)$.

\begin{definition}[Convex hull of parabolic boundary values]\label{def:parbdvalues}
Let $\bm u\in L^p((0,T); W^{1,p}(\Omega; \R^N))\cap C([0,T];L^2(\Omega;\R^N))$. Then we define $\convhull \bm u(\parbd Q_T)\subset\R^N$ to be the smallest closed convex set such that $\bm u(t,x)\in \convhull \bm u(\parbd Q_T)$ for a.e. $(t,x)\in (0,T)\times \partial \Omega$ and $\bm u(0,x)\in \convhull \bm u(\parbd Q_T)$ for a.e. $x\in \Omega$. Note that this is well-defined since by assumption $\bm u(0,\cdot)\in L^2(\Omega;\R^N)$ and the trace of $\bm u$ is in $L^p((0,T)\times \partial \Omega;\R^N)$, as $p<\infty$.
\end{definition}

\begin{remark}
Similarly as in Remark \ref{rem:essbdval-ell}, one can alternatively define $\convhull \bm u(\parbd Q_T)$ as the closed convex hull of essential values of $\bm u$ on $\parbd Q_T$.
\end{remark}

\begin{theorem}[Convex hull property for certain nonlinear parabolic systems]\label{thm:CHPparanonlin}
Let $1<p<\infty $ and let $\bm u\in L^{p'}((0,T); W^{1,p}(\Omega; \R^N))\cap C([0,T];L^2(\Omega;\R^N))$ with $\partial_t\bm u\in L^p((0,T); W^{-1,p'}(\Omega;\R^N))$ be a weak solution to (\ref{eqn:parnonlin}). Let the coefficients satisfy that $\bm a$ is uniformly positive definite, so that there exists $\lambda>0$ such that for a.e. $(t,x)\in (0,T) \times \Omega$ we have 
$$    \xi\cdot \bm a(t,x) \xi \geq \lambda\abs{\xi}^2, \quad \xi\in\R^n,$$
and that for some $C>0$ it holds
\begin{equation}\label{eqn:b-growth}
\abs{\bm b(t,x,z,\xi)}\leq C\sqrt{a_0(t,x,z,\xi)},\quad (t,x,z,\xi)\in (0,T)\times \Omega\times \R^N\times\R^{N\times n}.
\end{equation}
Then $\bm u$ satisfies the following convex hull property:
\begin{enumerate}[(i)]
    \item If $c=0$ a.e. in $\Omega$ and $K=\convhull \bm u(\parbd Q)$, then $\bm u(x,t) \in K$ for a.e. $(t,x)\in(0,T)\times \Omega$. \smallskip
    \item If $c\geq 0$ a.e. in $\Omega$ and $K=\convhull (\bm u(\parbd Q)\cup \{0\} )$, then $\bm u(x,t) \in K$ for a.e. $(t,x)\in(0,T)\times \Omega$.
\end{enumerate}
\end{theorem}

\begin{proof}
We use as a test function in \eqref{eqn:wspara} the function $\bm u(t)-\Pi_K \bm u(t)\in W^{1,p}_0(\Omega;\R^N)\cap L^2(\Omega;\R^N)$. Therefore we have that for a.e. $t\in (0,T)$
\begin{multline}
    \underbrace{\scalar{\partial_t \bm u(t)}{\bm u(t)-\Pi_K \bm u(t)}}_{(t)}+\underbrace{\int_\Omega \mathbb A \nabla \bm u(t) :\nabla (\bm u(t)-\Pi_K \bm u(t)\dx)}_{(a)} \\ +\underbrace{\int_\Omega \bm b \cdot \nabla \bm u(t)\cdot (\bm u(t)-\Pi_K \bm u(t))\dx}_{(b)} +\underbrace{\int_\Omega c\bm u(t)\cdot(\bm u(t)-\Pi_K \bm u(t))\dx}_{(c)}=0.
\end{multline}
  Now let us consider each term separately. Using Proposition~\ref{prop:timederiv} we find that 
\begin{equation*}
  (t)=\scalar{\partial_t\bm u(t)}{\bm u(t)-\Pi_K \bm u(t)} = \frac12 \partial_t \int_\Omega |\bm u(t)-\Pi_K \bm u(t)|^2\dx.
\end{equation*}
For the $(a)$ term we can do the same trick with orthogonal change of variables as before. So for any $(t,x)\in Q_T$ we find an orthogonal matrix $Q(t,x)\in \R^{n\times n}$ and diagonal matrix $D(t,x)\in \R^{n\times n}$ with diagonal entries $d_i(t,x)\geq \lambda$ such that $\bm a(t,x)= Q^T(t,x)D(t,x)Q(t,x)$. Now for fixed $(t_0,x_0)$ put $\overline{Q}=Q(t_0,x_0)$. Denoting  $\overline{\bm u}(y,t)=\bm u(x,t)$ for $y=\overline{Q}x$ and following exactly the same computations as in the proof of Theorem \ref{thm:CHPellipticlin}, we get  
\begin{multline*}
    \mathbb A(t_0,x_0,\bm u(t_0,x_0),\nabla \bm u(t_0,x_0)) \nabla(\bm u(t_0,x_0) - \piku(t_0,x_0)): \nabla(\bm u(t_0,x_0) -\piku(t_0,x_0)) \\ =  a_0 (t_0,x_0,\bm u(t_0,x_0),\nabla \bm u(t_0,x_0))\cdot \hspace{25em} \\ \cdot\sum_{\alpha = 1}^N\sum_{i,j,k=1}^n\overline{q}_{ki}d_k(t_0,x_0) \overline{q}_{kj} \pd{(\bm u(t_0,x_0) -\piku(t_0,x_0) )^\alpha}{x_i}\pd{(\bm u(t_0,x_0) -\piku(t_0,x_0) )^\alpha}{x_j} \\= a_0 (t_0,x_0,\bm u(t_0,x_0),\nabla \bm u(t_0,x_0))\sum_{k=1}^n d_k(t_0,x_0) \abs{\pd{(\overline{\bm u}(t_0,y_0) -\Pi_K \overline{\bm u}(t_0,y_0) )}{y_k}}^2 \\ \geq \lambda  a_0 (t_0,x_0,\bm u(t_0,x_0),\nabla \bm u(t_0,x_0))\abs{\nabla (\bm u(t_0,x_0)-\piku (t_0,x_0))}^2.
\end{multline*}
Now keeping $t_0$ fixed and integrating this inequality over $x_0$ gives
\begin{equation*}
    \int_\Omega \mathbb A \nabla(\bm u(t_0) - \piku(t_0)): \nabla(\bm u(t_0) -\piku(t_0)) \dx  \geq \lambda \int_\Omega  a_0 
\abs{\nabla (\bm u(t_0)-\piku (t_0))}^2\dx.
\end{equation*}
By an analogous computation and Proposition \ref{prop:gradineq}, 
\begin{multline*}
    \mathbb A(t_0,x_0,\bm u(t_0,x_0),\nabla \bm u(t_0,x_0))\nabla \Pi_K \bm u(t_0,x_0) : \nabla (\bm u(t_0,x_0)-\Pi_K \bm u(t_0,x_0)) \\ =
    a_0 (t_0,x_0,\bm u(t_0,x_0),\nabla \bm u(t_0,x_0)) \sum_{k=1}^n  d_k(t_0,x_0)  \pd{\Pi_K \overline{\bm u}}{y_k}(t_0,y_0) \cdot \pd{(\overline{\bm u} - \Pi_K \overline{\bm u})}{y_k}(t_0,y_0)\geq 0.
\end{multline*}
Again, integrating over $x_0\in \Omega$ gives 
\begin{equation*}
 \int_\Omega \mathbb A(t,\cdot,\bm u, \nabla \bm u) \nabla \piku(t) :\nabla( \bm u(t)-\piku(t)) \dx \geq 0.
\end{equation*}
Altogether we have that for a.e. $t\in(0,T)$
\begin{multline*}
 (a)=\int_\Omega \mathbb A \nabla(\bm u(t)-\piku(t)):\nabla(\bm u(t)-\piku(t))\dx + \int_\Omega \mathbb A\nabla \piku(t) :\nabla( \bm u(t)-\piku(t)) \dx \\ \geq \lambda \int_\Omega a_0 (t,\cdot, \bm u(t), \nabla \bm u(t))  \abs{\nabla(\bm u(t)-\piku(t))}^2\dx.
\end{multline*}

Using the growth condition \eqref{eqn:b-growth}, Proposition \ref{prop:projorth}, the Cauchy-Schwarz and Young inequalities we obtain for every $\epsilon>0$:
\begin{align*}
(b)&= \int_\Omega \bm b \cdot \nabla (\bm u(t)-\piku(t))\cdot (\bm u(t)-\Pi_K \bm u(t))\dx +\int_\Omega \bm b \cdot \nabla \piku(t)\cdot (\bm u(t)-\Pi_K \bm u(t)) \dx \\
    &\geq -\int_\Omega \abs{\bm b} \abs{\nabla (\bm u(t)-\piku(t))}\abs{\bm u(t)-\piku(t)}\dx
    \\
    &\quad +\int_\Omega \sum_{i=1}^n b_i \underbrace{\pd{\piku(t)}{x_i}\cdot (\bm u(t)-\Pi_K \bm u(t))}_{=0\text{ (Proposition \ref{prop:projorth})}}\dx \\
    &\geq -C\int_\Omega \sqrt{a_0}\abs{\nabla (\bm u(t)-\piku(t))}\abs{\bm u(t)-\piku(t)}\dx  \\
    &\geq  -C\int_\Omega \epsilon  a_0 \abs{\nabla(\bm u(t)-\piku(t))}^2\dx -C\int_\Omega\frac{1}{4\epsilon}\abs{\bm u(t)-\piku(t)}^2\dx.
\end{align*}
Finally, for the last term, we distinguish the two cases from the statement of the theorem.

\begin{enumerate}[(i)]
    \item $c=0$ a.e., so clearly $(c) = 0$.
    \item $c\geq 0$ a.e. Then, since we have $0\in K$, we can use Corollary \ref{cor:projineq} to obtain
    \begin{multline*}
        (c)=\int_\Omega c(\bm u(t)-\piku(t))\cdot(\bm u(t)-\piku(t)) +c(\piku(t)-0)\cdot(\bm u(t)-\piku(t))\dx\\
        \geq \int_\Omega c \abs{\bm u(t)-\piku(t)}^2\dx.
    \end{multline*}
\end{enumerate}
In both cases $(c)\geq 0$. 

So, putting it all together, we obtain for a.e. $t\in(0,T)$ 
\begin{multline*}
    \frac 12 \partial_t \int_\Omega \abs{\bm u(t)-\piku(t)}^2\dx+(\lambda-C\epsilon)\int_\Omega a_0\abs{\nabla (\bm u(t)-\piku(t))}^2\dx \leq \frac{C}{4\epsilon}\int_\Omega \abs{\bm u(t)-\piku(t)}^2\dx.
\end{multline*}
Upon setting $\epsilon \leq \lambda/C$, so that $\lambda-C\epsilon \geq 0$, and denoting $\eta (t):= \int_\Omega \abs{\bm u(t)-\piku(t)}^2\dx$ we obtain
\begin{equation*}
    \eta'(t)\leq \frac{C}{2\epsilon}\eta(t)\quad \text{for a.e. }t\in(0,T).
\end{equation*}
Since $\eta(0)=0$, the Gronwall inequality implies $\eta(t)=0$, $t\in(0,T)$. Thus $\norm{\bm u(t)-\piku(t)}_{L^2(\Omega;\R^N)}=0$ for a.e. $t$, in particular $\bm u(t)=\piku(t)$ a.e. in $\Omega$  for a.e. $t\in(0,T)$, which means that $\bm u(x,t)\in K$ for a.e. $t\in(0,T)$ and a.e. $x\in\Omega$. This finishes the proof.
\end{proof}

\begin{remark}[On existence]
In our discussion we do not touch the question of the \emph{existence} of a solution. For this reason, the assumptions on the coefficients are rather broad, and in no way guarantee the existence of a solution, nor that the integrals in the weak formulation exist. We instead take the existence of a weak solutions as an \emph{assumption}, so we implicitly assume that all the nonlinear expressions in the weak formulation have the necessary integrability. We choose this approach so that the question of existence can be investigated independently of the present discussion.

Indeed the existence has been studied in other works, usually the requirement is some kind of monotonicity. We refer the interested reader to the monograph \cite[Part II]{roubicekNonlinearPartialDifferential2013} and the references within.
\end{remark}

\section{Counterexamples for linear systems}\label{chap:counterex}

Here we show two counterexamples to the convex hull property. They illustrate that coupling conditions for the coefficients cannot be relaxed.

\subsection{Elliptic system}
For elliptic linear systems 
\begin{equation*}
-\div (\mathbb A \nabla \bm u)=\bm 0,
\end{equation*}
we present a counterexample in one variable with $x$-dependent $\mathcal A(x)\in \R^{N\times N}$. 
\begin{example}\label{ex:counterexell}
Let $n=1$, $N=2$, $\mathcal A(x)=\left(\begin{smallmatrix} 1 & -x \\ -x & 1 \end{smallmatrix} \right)$, $\bm a\equiv 1$. Then we have $\mathcal A$ symmetric uniformly positive definite on $\Omega=(0,\ell)$ with $\ell<1$. The system of equations
\begin{align*}
-(u^1{}' -xu^2{}')' &= 0 \\
-(-xu^1{}' +u^2{}')' &= 0
\end{align*}
is uniformly elliptic and has a solution
\begin{align*}
u^1(x)=\log\left(\frac{1-x}{1+x}\right), \quad
u^2(x)=\log(1-x^2).
\end{align*}
Indeed, 
\begin{align*}
-(u^1{}'(x) -xu^2{}'(x))' &= -\left(\frac{-2}{1-x^2}-x\frac{-2x}{1-x^2} \right)'=0, \\
-(-x u^1{}'(x) +u^2{}'(x))'&= -\left(-x\frac{-2}{1-x^2}+\frac{-2x}{1-x^2} \right)'=0.
\end{align*}

As $\bm u(0)=(0,0)$ and $u^2(\ell)\neq 0$, then the convex hull property $\bm u((0,\ell))\subset \convhull \{ \bm u(0),\bm u(\ell) \}$ implies that the ratio $u^1(x)/u^2(x)$ is constant (and equal to $u^1(\ell)/u^2(\ell)$). However, by calculation
\begin{equation*}
\left( \frac{u^1(x)}{u^2(x)}\right)' = \frac{2 (1+x) \log (1+x)+2 (1-x) \log (1-x)}{\left(x^2-1\right) (\log (1-x)+\log (x+1))^2} \neq 0 \quad x\in(0,\ell),
\end{equation*}
therefore the ratio $u^1(x)/u^2(x)$ is non-constant and $\bm u$ does not satisfy the convex hull property.
\end{example}

\begin{remark}[On failure of variable projection]
Here it may be worthwhile to note where the method of projection to $K$ fails for $x$-dependent $\mathcal A$. Using at each $x\in\Omega$ the projection $\Pi_K^{\mathcal A(x)}$ is a plausible try. However, the estimate analogous to Proposition \ref{prop:gradineqA}, that is
\begin{equation*}
    \mathcal A(x)\nabla \bm u(x) :\nabla \Pi_K^{\mathcal A(x)} \bm u(x) \geq  \mathcal A(x) \nabla \Pi_K^{\mathcal A(x)} \bm u(x)  :\nabla \Pi_K^{\mathcal A(x)} \bm u(x) ,
\end{equation*}
cannot hold. To see this, take $\bm u$ constant near $x_0$, $K$ flat near $\Pi_K^{\mathcal A(x_0) }\bm u(x_0)$ and $\mathcal A$ variable near $x_0$ in such a way that $\nabla \Pi_K^{\mathcal A(x_0)} \bm u(x_0)\neq 0$. Since $\nabla \bm u(x_0)=0$, this contradicts the inequality.
\end{remark}

\subsection{Parabolic system}
In the case of linear parabolic systems
\begin{equation*}\label{eqn:paralinsys}
    \partial_t \bm u-\div(\mathbb A\nabla \bm u)=\bm 0
\end{equation*}
we present a counterexample with constant coefficients, in one spatial dimension and $\mathcal A$ diagonal.
\begin{example}\label{ex:counterexpara}
Consider $n=1$, $N=2$, $\mathcal A=\left(\begin{smallmatrix}
a_1 & 0 \\
0 & a_2\\
\end{smallmatrix}\right)$ for some constants $a_1,a_2>0$, $a_1\neq a_2$ and the following problem:
\begin{align*}
    \partial_t u^1 - a_1\pd{^2 u^1}{x^2}&=0 \quad\text{in }Q=(0,\infty)\times(0,\pi),\\
    \partial_t u^2 - a_2 \pd{^2 u^2}{x^2}&=0 \quad\text{in }Q.
\end{align*}
Then a solution is 
\begin{equation*}
    u^1(t,x)= e^{-a_1 t}\sin x, \quad u^2(t,x)=e^{-a_2 t}\sin x, \quad (t,x\\)\in [0,\infty) \times [0,\pi],
\end{equation*}
which has the parabolic boundary values 
\begin{align*}
u^1(0,x)=u^2(0,x)&=\sin x, \quad x\in(0,\pi),\\
    u^1(t,0)= u^1(t,\pi)=u^2(t,0)=u^2(t,\pi)&=0, \quad t\in (0,\infty).
\end{align*}
Now the closed convex hull of the boundary values is $\convhull \bm u(\parbd Q) =\{(s,s):s\in[0,1]\}$, whereas for $t>0$, $x\in (0,\pi)$ we see that $u^1(t,x)\neq u^2(t,x)$, so $\bm u(t,x)\notin \convhull \bm u(\parbd Q)$. Therefore $\bm u$ does not satisfy the convex hull property on any finite time interval. 

\end{example}

\subsection*{Acknowledgements}
I would like to thank Lars Diening and Sebastian Schwarzacher for introducing me to the problem, their ideas and discussions. I acknowledge the support of the the ERC-CZ grant LL2105, the Czech Science Foundation (GA\v{C}R) under grant No.\@ 23-04766S, and Charles University, project GA UK No.\@ 393421.

\bibliographystyle{alpha}
\bibliography{biblio}
\end{document}